\documentclass[12pt]{article}
\usepackage{amsmath,amsthm}
\usepackage{amsfonts}
\usepackage{amssymb}
\usepackage{url}
\usepackage{graphicx}
\usepackage{xspace}
\usepackage{epstopdf}

\usepackage{algorithmic,algorithm}
\usepackage{subfigure}

\usepackage{color}
\usepackage[normalem]{ulem}

\numberwithin{equation}{section}

\newtheorem{theorem}{Theorem}[section]
\newtheorem{lemma}{Lemma}[section]
\newtheorem{proposition}{Proposition}[section]

\newtheorem{conjecture}{Conjecture}[section]

\theoremstyle{definition}
\newtheorem{example}{Example}[section]

\theoremstyle{remark}

\DeclareMathOperator*{\Argmin}{arg\,min}
\DeclareMathOperator{\trace}{tr}
\DeclareMathOperator{\val}{val}

\def\real{\mathbb{R}}

\newcommand{\mainalg}{\text{FO}\xspace}
\newcommand{\gradalg}{\text{GM}\xspace}
\newcommand{\fastgradalg}{\text{FGM}\xspace}
\newcommand{\fastgradalgb}{\text{FGM$^\prime$}\xspace}

\newcommand{\comment}[1]{}

\title{Performance of first-order methods for smooth convex minimization: a novel approach}

\author{Yoel Drori and Marc Teboulle\thanks{School of Mathematical Sciences, Tel-Aviv University, Ramat-Aviv 69978,
Israel ({\tt dyoel@post.tau.ac.il, teboulle@math.tau.ac.il})}}

\begin{document}
\maketitle

\begin{abstract}
We introduce a novel approach for analyzing the performance of first-order black-box
optimization methods. We focus on smooth unconstrained convex minimization over the Euclidean
space $\real^d$. Our approach relies on the observation that by definition, the worst case
behavior of a black-box optimization method is by itself an optimization problem, which we
call the Performance Estimation Problem (PEP). We formulate and analyze the PEP for two
classes of first-order algorithms. We first apply this approach on the classical gradient
method and derive a new and tight analytical bound on its performance. We then consider a
broader class of first-order black-box methods, which among others, include the so-called
heavy-ball method and the fast gradient schemes. We show that for this broader class, it is
possible to derive new numerical bounds on the performance of these methods by solving an
adequately relaxed convex semidefinite PEP. Finally, we show an efficient procedure for
finding optimal step sizes  which results in a first-order black-box method that achieves best
performance.
\medskip

\noindent{\bf Keywords:} Performance of First-Order Algorithms, Rate of Convergence,
Complexity, Smooth Convex Minimization, Duality, Semidefinite Relaxations, Fast Gradient
Schemes, Heavy Ball method.

\end{abstract}

\section{Introduction}
First-order convex optimization methods have recently gained in popularity both in theoretical
optimization and in many scientific applications, such as signal and image processing,
communications, machine learning, and many more. These problems are very large scale, and
first-order methods, which in general involve very cheap and simple computational iterations,
are often the best option to tackle such problems in a reasonable time, when moderate accuracy
solutions are sufficient. For convex optimization problems, there exists an extensive
literature on the development and analysis of first-order methods, and in recent years, this
has been revitalized at a quick pace due to the emergence of many fundamental new applications
alluded above, see e.g., the recent collections \cite{convex-sig-10, opt-ml-11} and references
therein.



This work is not on the development of new algorithms, rather it focuses on the theoretical
performance analysis of first order methods for unconstrained minimization with an objective
function which is known to belong to a given family $\mathcal{F}$ of smooth convex functions
over the Euclidean space $\real^d$, the function itself is not known.

Following the seminal work  of Nemirovski and Yudin \cite{nemi-yudi-book83} in the complexity
analysis of convex optimization methods, we measure the computational cost based on the oracle
model of optimization. According to this model, a \emph{first-order black-box} optimization
method is an algorithm which has knowledge of the underlying space $\real^d$ and the family
$\mathcal{F}$, the function itself is not known. To gain information on the function to be
minimized, the algorithm  queries a first order oracle, that is, a subroutine which given as
input a point in $\real^d$, returns the value of the objective function and its gradient at
that point. The algorithm thus generates a finite sequence of points $\{x_i\in \mathbb{R}^d :
i=0,\dots,N\}$, where at each step the algorithm can depend only on the previous steps, their
function values and gradients via some rule
$$
 x_0 \in \real^d,\; x_{i+1}={\cal A}(x_0,\dots,x_i; f(x_0),\dots,f(x_i); f'(x_0),\dots, f'(x_i)), \; i=0, \ldots, N-1,
$$
where $f'(\cdot)$ stands for the gradient of $f(\cdot)$.  Note that the algorithm has another
implicit knowledge, i.e., that the distance from its initial point $x_0$ to a minimizer
$x_\ast \in X_\ast(f)$ of $f$ is bounded by some constant $R>0$, see more precise definitions
in the next section.

Given a desired accuracy $\varepsilon >0$, applying the given algorithm on the function $f$ in
the class $\mathcal{F}$, the algorithm stops when it produces an approximate solution
$x_\varepsilon$ which is $\varepsilon$-optimal, that is such that $$f(x_\epsilon) - f(x_\ast)
\leq \epsilon.$$ The performance (or complexity) of a first order black-box optimization
algorithm is then measured by the number of oracle calls  the algorithm needs to find such an
approximate solution. Equivalently, we can measure the performance of an algorithm by looking
at the absolute inaccuracy $$ \delta(f,x_N) = f(x_N) - f(x_\ast),$$ where $x_N$ is the result
of the algorithm after making $N$ calls to the oracle. Throughout this paper we will use the
latter form to measure the performance  of a given algorithm.

Building on this  model, in this work  we introduce a novel approach for analyzing the
performance of a given first order scheme. Our approach relies on the observation that by
definition, the worst case behavior of a first-order black-box optimization algorithm
 is by itself an optimization problem which consists of finding the maximal
absolute inaccuracy over all possible inputs to the algorithm.
Thus, with 
$x_N$  being the output
of the algorithm after making $N$ calls to the oracle, 
we look at the solution of the following {\em Performance Estimation Problem} (PEP):
\begin{align*}
    \begin{aligned}
    \max \quad & f(x_N)-f(x_\ast) \\
    \text{s.t. }\quad
                        & f \in \mathcal{F},\\
                        & x_{i+1}={\cal A}(x_0,\dots,x_i; f(x_0),\dots,f(x_i); f'(x_0),\dots, f'(x_i)),
                        \; i=0, \ldots, N-1,\\
                        & x_\ast \in X_\ast(f), \|x_\ast-x_0\|\leq R\\
                        & x_0, \ldots, x_N, x_\ast \in\mathbb{R}^d.
    \end{aligned}
    \tag{P}\label{P:baseproblemIntro}
\end{align*}

At first glance this problem seems very hard or impossible to solve.  We overcome this
difficulty through an analysis that relies on various types of relaxations, including duality
and semi-definite relaxation techniques.
The problem and setting, and an outline of the
underlying idea of the proposed approach for analyzing \eqref{P:baseproblemIntro} are described in Section~\ref{S:mainapproach}.
In order to develop the basic idea and tools underlying our proposed approach, we first focus on the
fundamental gradient method (GM) for smooth convex minimization, and then extend it to a broader class of first order black
box minimization methods.  Obviously, the gradient method is a particular case of this broader class that will be analyzed below.
However,  it is quite important to start with the gradient method for two reasons. First, it allows to acquaint the reader
in a more transparent way with the techniques and methodolgy we need to develop  in order to analyze (PEP),  thus
paving the way to tackle more general schemes. Secondly, for the gradient method,
we are able to prove a new and tight bound on its performance which is given {\em analytically}, see Section~\ref{S:analyticalbound}.
Capitalizing on the methodology and tools developed in the past section,  in Section~\ref{S:numerical_bounds}, we  consider a broader
class of first-order black-box methods, which among others, is shown to include the so-called
heavy-ball  \cite{polyak1964} and fast gradient schemes  \cite {nest-book-04}. Although an
analytical solution is not available for this general case, we show that for this broader
class of methods, it is always possible to compute {\em numerical bounds} for an adequate
relaxation of the corresponding PEP,  allowing to derive new  bounds on the performance of
these methods.  We then derive in Section~\ref{S:optalg} an efficient procedure for finding
optimal step sizes which results in a first-order method that achieves best performance. Our
approach and analysis give rise to some interesting problems leading us to suggest some
conjectures.  Finally, an appendix includes the proof of a technical result.

\paragraph{Notation.} For a differentiable function $f$, its gradient at $x$ is denoted by
$f'(x)$. The Euclidean norm of a vector $x\in\real^d$ is denoted as $\|x\|$.
The set of symmetric matrices in $\real^{n\times n}$ is denoted by $\mathbb{S}^n$.
For two symmetric
matrices $A$ and $B$, $A\succeq B, (A \succ B)$ means $A-B \succeq 0 \,(A-B \succ 0)$ is
positive semidefinite (positive definite).
We use $e_i$ for the $i$-th canonical basis vector in $\mathbb{R}^N$, which consists of all zero components,
except for its $i$-th entry which is equal to one,
and use $\nu$ to denote a unit vector in $\mathbb{R}^d$. For an optimization problem (P), $\val(P)$
stands for its optimal value.

\section{The Problem and the Main Approach}\label{S:mainapproach}

\subsection{The Problem and Basic Assumptions}

Let $\mathcal{A}$ be a first-order algorithm for solving the optimization problem
\begin{align*}
    \begin{aligned}
    (M)\quad \min\{ f(x):\, x \in \real^d\}.
    \end{aligned}
\end{align*}

Throughout the paper we make the following assumptions:
\begin{itemize}
\item $f:\real^d \to \real$  is a convex function of the type $C^{1,1}_L(\real^d)$, i.e.,
continuously differentiable with Lipschitz continuous gradient:
\[
    \|f'(x)-f'(y)\|\leq L \|x-y\|,\,\,\forall x,y\in \real^d,
\]
where $L>0$ is the Lipschitz constant.
\item We assume that (M) is solvable, i.e., the optimal set $X_*(f):=\Argmin f$ is nonempty, and for
$x_*\in X_*(f)$ we set $f^*:=f(x_*).$
\item There exists $R>0$, such that
the distance from $x_0$ to an optimal solution $x_\ast\in X_*(f)$ is bounded by
$R$.\footnote{In general, the terms $L$ and $R$ are unknown or difficult to compute, in which case
some upper bound estimates can be used in place. Note that all currently known complexity
results for first order methods depend on $L$ and $R$.}
\end{itemize}


Given a convex function $f$ in the class $C^{1,1}_L(\real^d)$ and any starting point $x_0 \in
\real^d$, the algorithm $\mathcal{A}$ is a first-order black box scheme, i.e., it is allowed
to access $f$ only through the sequential calls to the first order oracle that returns the
value and the gradient of $f$ at  any input point $x$.   The algorithm  $\mathcal{A}$ then
generates a sequence of points 
$x_i \in \real^d, \;i=0,\ldots, N$.

\subsection{Basic Idea and Main Approach}
We are interested in measuring the worst-case behavior of a given algorithm $\mathcal{A}$ in
terms of the absolute inaccuracy  $f(x_N) - f(x_\ast)$, by solving problem
\eqref{P:baseproblem} defined in the introduction, namely
\begin{align*}
    \begin{aligned}
    \max \quad & f(x_N)-f(x_\ast) \\
    \text{s.t. }\quad
                        & f \in C^{1,1}_L(\mathbb{R}^d) \text{, $f$ is convex},\\
                        & x_{i+1}={\cal A}(x_0,\dots,x_i; f(x_0),\dots,f(x_i); f'(x_0),\dots, f'(x_i)), \; i=0,
                        \ldots, N-1,\\
                        & x_\ast \in X_\ast(f), \|x_\ast-x_0\|\leq R,\\
                        & x_0, \ldots, x_N, x_\ast \in\mathbb{R}^d.
    \end{aligned}
    \tag{P}\label{P:baseproblem}
\end{align*}

To tackle this problem, we suggest to perform a series of relaxations 
thereby reaching a tractable optimization problem. 
\medskip

A main difficulty in problem \eqref{P:baseproblem} lies in the functional constraint
($f$ is a convex function in $C^{1,1}_L(\mathbb{R}^d)$),
i.e., we are facing an abstract hard optimization problem in infinite dimension. To overcome
this difficulty,  the approach taken in this paper is to {\em relax} this constraint so that
the problem can be reduced and formulated as an explicit finite dimensional problem that can
eventually be adequately analyzed.

An un-formal description of the underlying idea consists of two main steps as follows:

\begin{itemize}
\item Given an algorithm ${\cal A}$ that generates a finite sequence of points,
to build a problem in finite dimension we
replace the functional constraint $f\in C^{1,1}_L$ in \eqref{P:baseproblem} by new variables in $\real^d$. These
variables, are the points $\{x_0, x_1, \ldots x_N, x_\ast \}$ themselves, the function values
and their gradients at these points. Roughly speaking, this can be seen as a sort of
discretization of $f$ at a selected set of points.
\item To define constraints that relate the new variables, we use relevant/useful
properties characterizing the family of convex functions in $C^{1,1}_L$, as well as the
rule(s) describing the given algorithm ${\cal A}$.
\end{itemize}

This approach can, in principle, be applied to any optimization algorithm. Note that any
relaxation performed on the maximization problem \eqref{P:baseproblem} may increase its
optimal value, however, the optimal value of the relaxed problem still remains a valid upper
bound on $f(x_N)-f^\ast$.

A formal description on how this approach can be applied to the gradient method is described
in the next section, which as we shall see, allows us to derive a new tight bound on the
performance of the gradient method.

\section{An Analytical Bound for the Gradient Method}\label{S:analyticalbound}

 To develop the basic idea and tools underlying the proposed approach
for analyzing the performance of iterative optimization algorithms, in this section we focus
on the simplest fundamental method for smooth convex minimization, the {\em Gradient Method}
(GM). It will also pave the way to tackle more general first-order schemes as developed in the
forthcoming sections.

\subsection{A Performance Estimation Problem for the Gradient Method}
Consider the gradient algorithm with constant step size, as applied to problem $(M)$, which generates a sequence of
points as follows:\medskip

\fbox{\parbox{14cm}{ {\bf Algorithm GM} \medskip
\begin{enumerate}
\setcounter{enumi}{-1}
\item Input: $f\in C^{1,1}_L(\mathbb{R}^d)$ convex, $x_0\in \mathbb{R}^d.$
\item For $i=0,\dots,N-1$, compute $x_{i+1} = x_i - \frac{h}{L} f^\prime(x_i)$.
\end{enumerate}
}}
\bigskip

Here $h>0$ is fixed.
At this point, we recall that for $h=1$, the convergence rate of the GM algorithm can be shown
to be (see for example \cite{nest-book-04,beck-tebo-fista09}):
\begin{equation}\label{E:oldgradientbound}
    f(x_N)-f^\ast \leq \frac{L\|x_0-x_\ast\|^2}{2N}, \quad \forall x_\ast \in X_\ast(f).
\end{equation}
\medskip

To  begin our analysis, we first recall a fundamental well-known property for the class of
convex $C^{1,1}_L$ functions, see e.g., \cite[Theorem~2.1.5]{nest-book-04}.

\begin{proposition}\label{prop:lip} Let $f:\real^d \to \real$ be convex and $C^{1,1}_L$. Then,
\begin{equation}\label{E:basicconvexprop}
    \frac{1}{2L} \|f^\prime(x)-f^\prime(y)\|^2 \leq f(x)-f(y)-\langle f^\prime(y), x-y\rangle, \;\mbox{for all}
    \; x,y\in \mathbb{R}^d.
\end{equation}
\end{proposition}

Let $x_0\in \real^d$ be any starting point, let $\{x_1,\dots,x_N \}$ be the points generated
by Algorithm (GM) and let $x_\ast$ be a minimizer of $f$. Applying \eqref{E:basicconvexprop}
on the points $\{x_0,\dots,x_N, x_\ast\}$, we get:
\begin{align}\label{E:basicconstraints}
    & \frac{1}{2L} \|f^\prime(x_i)-f^\prime(x_j)\|^2 \leq f(x_i)-f(x_j)-\langle f^\prime(x_j), x_i-x_j\rangle, \quad i,j=0,\dots,N,\ast.
\end{align}
Now define
\begin{align*}
    \delta_i &:= \frac{1}{L\|x_\ast-x_0\|^2}(f(x_i)-f(x_\ast)),\quad i=0,\dots,N, \ast \\
    g_i &:= \frac{1}{L\|x_\ast-x_0\|}f^\prime(x_i),\quad i=0,\dots,N, \ast
\end{align*}
and note that  we always have $\delta_\ast=0$ and $g_\ast=0$.

In terms of $\delta_i$, $g_i$,
condition \eqref{E:basicconstraints} becomes
\begin{equation}\label{carf}
    \frac{1}{2}\| g_i - g_j \| ^2 \leq \delta_i - \delta_j - \frac{\langle g_j, x_i-x_j
    \rangle}{\|x_\ast-x_0\|}, \quad i,j=0,\dots,N,\ast,
\end{equation}
and the recurrence defining (GM) reads:
$$
x_{i+1}=x_i - h \|x_\ast -x_0\| g_i, \quad i=0, \ldots , N-1 .
$$

 Problem \eqref{P:baseproblem} can now be
relaxed by \emph{discarding} the underlying function $f \in C^{1,1}_L$ in
\eqref{P:baseproblem}. That is, the constraint in the function space $f\in C^{1,1}_L$ with $f$
convex, is replaced by the inequalities (\ref{carf}) characterizing this family of functions
and expressed in terms of the variables $x_0,\dots,x_N,x_\ast\in\mathbb{R}^d$,
$g_0,\dots,g_N\in\mathbb{R}^d$ and $\delta_0,\dots,\delta_N\in\mathbb{R}$ generated by (GM).
Thus, an upper bound on the worst case behavior of $f(x_N)-f(x_\ast)= L \|x_\ast -x_0\|^2
\delta_N$
 can be obtained by solving the following relaxed PEP:
\begin{align*}
    \max_{\substack{x_0,\dots,x_N,x_\ast\in\mathbb{R}^d,\\g_0,\dots,g_N\in\mathbb{R}^d,\\\delta_0,\dots,\delta_N \in\mathbb{R}}}\ & L\|x_\ast-x_0\|^2 \delta_N \\
    \text{s.t. }\quad
                        & x_{i+1} = x_i - h \|x_\ast-x_0\| g_i,\quad i=0,\dots,N-1,\\
                        & \frac{1}{2}\| g_i - g_j \| ^2 \leq \delta_i - \delta_j - \frac{ \langle g_j, x_i-x_j \rangle}{\|x_\ast-x_0\|},\quad  i,j=0,\dots,N,\ast, \\
                        &\|x_\ast-x_0\|\leq R.
\end{align*}

\paragraph{Simplifying the PEP} The obtained problem remains nontrivial to tackle. We will now
perform some simplifications on this problem that will be useful for the forthcoming analysis.

First, we observe that the problem is invariant under the transformation $g_i^\prime
\leftarrow Q g_i$, $x_i^\prime \leftarrow Q x_i$ for any orthogonal transformation $Q$. We can
therefore assume without loss of generality that $x_\ast- x_0= \|x_\ast-x_0\|\nu$, where $\nu$
is any given unit vector in $\mathbb{R}^d$. Therefore, for $i=\ast$ the inequality constraints
reads
\[
    \frac{1}{2}\| g_\ast - g_j \| ^2 \leq \delta_\ast - \delta_j - \frac{ \langle g_j, \|x_\ast-x_0\|\nu+x_0-x_j \rangle}{\|x_\ast-x_0\|},\quad     j=0,\dots,N.
\]
Secondly,
we consider \eqref{carf} for the four cases $i=\ast$, $j=\ast$, $i<j$ and $j<i$,
and use the equality constraints
$$
x_{i+1} = x_i - h \|x_\ast-x_0\| g_i,\quad i=0,\dots,N-1
$$
to eliminate the variables $x_1,\dots,x_N$.
After some algebra, we reach the following form for the PEP:
\begin{align*}
    \begin{aligned}
    \max_{x_0,x_\ast,g_i \in\mathbb{R}^d,\delta_i \in\mathbb{R}}\ & L\|x_\ast-x_0\|^2 \delta_N \\
    \text{s.t. }
                        & \frac{1}{2} \| g_i - g_j \| ^2 \leq \delta_i - \delta_j - \langle g_j, \sum_{t=i+1}^{j} h g_{t-1} \rangle,\quad i<j=0,\dots,N, \\
                        & \frac{1}{2} \| g_i - g_j \| ^2 \leq \delta_i - \delta_j + \langle g_j, \sum_{t=j+1}^{i} h g_{t-1} \rangle,\quad j<i=0,\dots,N, \\
                        & \frac{1}{2} \| g_i \| ^2 \leq \delta_i,\quad                                                                                          i=0,\dots,N, \\
                        & \frac{1}{2} \| g_i \| ^2 \leq - \delta_i - \langle g_i, \nu + \sum_{t=1}^{i} h g_{t-1} \rangle,\quad  i=0,\dots,N, \\
                        & \|x_\ast-x_0\| \leq R,
    \end{aligned}
\end{align*}
where $i<j=0,\dots,N$ is a shorthand notation  for $i=0,\dots,N-1, j= i+1,\dots, N$.

Finally, we note that the optimal solution for this problem is attained when
$\|x_\ast-x_0\|=R$, and hence we can also eliminate the variables $x_0$ and $x_\ast$. This
produces the following PEP for the gradient method, a {\em nonconvex quadratic} minimization
problem:
\begin{align*}
    \begin{aligned}
    \max_{g_i \in\mathbb{R}^d,\delta_i \in\mathbb{R}}\ & LR^2 \delta_N \\
    \text{s.t. }
                        & \frac{1}{2} \| g_i - g_j \| ^2 \leq \delta_i - \delta_j - \langle g_j, \sum_{t=i+1}^{j} h g_{t-1} \rangle,\quad i<j=0,\dots,N, \\
                        & \frac{1}{2} \| g_i - g_j \| ^2 \leq \delta_i - \delta_j + \langle g_j, \sum_{t=j+1}^{i} h g_{t-1} \rangle,\quad j<i=0,\dots,N, \\
                        & \frac{1}{2} \| g_i \| ^2 \leq \delta_i,\quad                                                                                                  i=0,\dots,N, \\
                        & \frac{1}{2} \| g_i \| ^2 \leq - \delta_i - \langle g_i, \nu + \sum_{t=1}^{i} h g_{t-1} \rangle,\quad  i=0,\dots,N. \\
    \end{aligned}
\end{align*}

This problem can be written  in a more compact and useful form. Let $G$ denote the
$(N+1)\times d$ matrix whose rows are $g_0^T,\dots g_N^T$, and  for notational convenience let
$u_i\in\mathbb{R}^{N+1}$ denote the canonical unit vector
\[
    u_i = e_{i+1}, \quad i=0,\dots,N.
\]
Then for any $i,j$, we have
\[
   g_i= G^T u_i,\; \trace (G^T u_i u_j^T G)=\langle g_i, g_j \rangle, \;\mbox{and}\; \langle G^T u_i,\nu \rangle=\langle g_i, \nu\rangle.
\]
Therefore, by defining the following $(N+1)\times (N+1)$ symmetric matrices
\begin{equation}\label{E:gradientmatrices}
\begin{aligned}
    A_{i,j} &= \frac{1}{2} (u_i-u_j)(u_i-u_j)^T + \frac{1}{2}\sum_{t=i+1}^{j} h (u_j u_{t-1}^T+u_{t-1} u_j^T),\\
    B_{i,j} &= \frac{1}{2} (u_i-u_j)(u_i-u_j)^T -  \frac{1}{2}\sum_{t=j+1}^{i} h (u_j u_{t-1}^T+u_{t-1} u_j^T),\\
    C_i &= \frac{1}{2} u_i u_i^T,\\
    D_i &= \frac{1}{2} u_i u_i^T+ \frac{1}{2}\sum_{t=1}^{i} h (u_i u_{t-1}^T + u_{t-1} u_i^T),
\end{aligned}
\end{equation}
we can express our nonconvex quadratic minimization problem in terms of
$\delta:=(\delta_0,\dots,\delta_N) \in \real^{N+1}$
and the new matrix variable $G\in \mathbb{R}^{(N+1)\times d}$ as follows
\begin{align*}
    \begin{aligned}
    \max_{G\in \mathbb{R}^{(N+1)\times d},\delta \in\mathbb{R}^{N+1}}\ &  LR^2 \delta_N \\
    \text{s.t. }
                        & \trace (G^T A_{i,j} G) \leq \delta_i - \delta_j,\quad    i<j=0,\dots,N, \\
                        & \trace(G^T B_{i,j} G) \leq \delta_i - \delta_j,\quad      j<i=0,\dots,N, \\
                        & \trace(G^T C_i G) \leq \delta_i,\quad                             i=0,\dots,N, \\
                        & \trace(G^T D_i G+ \nu u_i^T G ) \leq - \delta_i,\quad     i=0,\dots,N.
    \end{aligned}
    \tag{G}\label{P:gbound}
\end{align*}

Problem (G) is a nonhomogeneous {\em Quadratic Matrix Program}, a class of problems recently
introduced and studied by Beck \cite{beck-qm06}.

\subsection{A Tight Performance Estimate for the Gradient Method}\label{sub-tight}


We now proceed to establish the two main results of this section. First, we derive an upper
bound on the performance of the gradient method, this is accomplished via using duality
arguments. Then, we show that this bound can actually be attained by applying the gradient
method on a specific convex function in the class $C^{1,1}_L$.

In order to simplify
the following analysis,
we will remove some constraints from \eqref{P:gbound} and consider
the bound produced by the following relaxed problem:
\begin{align*}
    \begin{aligned}
    \max_{G\in \mathbb{R}^{(N+1)\times d},\delta \in\mathbb{R}^{N+1}}\ &  LR^2 \delta_N \\
    \text{s.t. }
                        & \trace (G^T A_{i-1,i} G) \leq \delta_{i-1} - \delta_i,\quad    i=1,\dots,N, \\
                        & \trace(G^T D_i G+ \nu u_i^T G ) \leq - \delta_i,\quad     i=0,\dots,N.
    \end{aligned}
    \tag{G$^\prime$}\label{P:gpr}
\end{align*}
As we shall show below, it turns out that this additional relaxation has no
damaging effects and produces the desired performance bound when $0< h\leq 1$.
\medskip

We are interested in deriving a  dual problem for \eqref{P:gpr} which is as simple as
possible, especially with respect to its dimension. As noted earlier, problem \eqref{P:gpr} is
a nonhomogeneous quadratic matrix program, and a dual problem for \eqref{P:gpr} could be
directly obtained by applying the results developed by Beck \cite{beck-qm06}. However, the
resulting obtained dual will involve an additional matrix variable $\Phi \in \mathbb{S}^d$,
where $d$ can be very large. Instead, here by exploiting the special structure of the second
set of nonhomogeneous inequalities given in \eqref{P:gpr}, we derive an alternative  dual
problem, but with only one additional variable $t \in \real$.

To establish our dual result, the next lemma shows that a dimension reduction is possible when
minimizing  a quadratic matrix function sharing the special form as the one that appears in
problem \eqref{P:gpr}.

\begin{lemma}\label{L:quadlemma}
Let $f(X)=\trace (X^T Q X + 2 ba^T X)$ be a quadratic function, where $X\in
\mathbb{R}^{n\times m}$, $Q\in \mathbb{S}^n$, $a\in \mathbb{R}^n$ and $0\neq b\in
\mathbb{R}^m$. Then
\[
    \inf_{X\in \mathbb{R}^{n\times m}} f(X) = \inf_{\xi\in \mathbb{R}^n} f(\xi b^T).
\]
\end{lemma}
\begin{proof}
First, we recall (this can be easily verified) that $\inf \{f(X): X \in \real^{n\times m}\}
>-\infty$ if and only if $Q \succeq 0$, and there exists at least one solution ${\bar X}$ such
that
\begin{equation}\label{s1}
 Q{\bar X}+ ab^T=0 \;
\Leftrightarrow\; {\bar X}^TQ +ba^T=0,
\end{equation}
i.e., the above is just $\nabla f (X)= 0$ and characterizes the minimizers of the convex
function $f(X)$. Using (\ref{s1}) it follows that $\inf_X f(X)=f({\bar X})=\trace (ba^T{\bar
X})$.
 Now, for any $\xi \in \mathbb{R}^n$, we have $f(\xi b^T)=\|b\|^2(\xi^TQ\xi + 2a^T\xi)$. Thus, likewise,
 $\inf\{f(\xi b^T): \xi \in \real^n\} >-\infty$ if and only if $Q \succeq 0$ and there exists ${\bar \xi}\in \real^n$ such
that \begin{equation}\label{s2} Q{\bar \xi} + a=0,\end{equation} and using (\ref{s2}) it
follows $\inf_\xi f(\xi b^T)=f({\bar \xi}b^T)=\|b\|^2a^T{\bar \xi}=\trace (ba^T{\bar
\xi}b^T)$. Now, using (\ref{s1})-(\ref{s2}), one obtains ${\bar X}^TQ=-ba^T$ and $Q({\bar X}-
{\bar \xi}b^T)=0$, and hence it follows that
\begin{eqnarray*}
f({\bar X})-f({\bar \xi})&=&\trace (ba^T({\bar X}- {\bar \xi}b^T))\\
                        &=&\trace (-{\bar X}^TQ({\bar X}- {\bar \xi}b^T)=0.
                        \end{eqnarray*}
\end{proof}

Equipped with Lemma \ref{L:quadlemma},  we now derive  a Lagrangian dual for problem
\eqref{P:gpr}.

\begin{lemma}\label{lem:dual} Consider problem \eqref{P:gpr} for any fixed $h\in \real$ and  $L,R >0$.
A Lagrangian dual of \eqref{P:gpr} is given by the following convex program:
\begin{align*}
    \begin{aligned}
        \min_{\lambda \in \real^N, t \in \real}\{\frac{1}{2} LR^2 t:\; \lambda \in \Lambda, \; S(\lambda, t) \succeq 0 \},
    \end{aligned}
    \tag{DG$^\prime$}\label{P:dualgpr}
\end{align*}
where $\Lambda:=\{\lambda \in \real^N:\; \lambda_{i+1}-\lambda_i \geq 0, \quad
i=1,\dots,N-1,\; 1-\lambda_N \geq 0,\; \lambda_i \geq 0,\quad i=1,\dots,N\}$,
the matrix
$S(\cdot,\cdot)\in \mathbb{S}^{N+2}$ is given by \[
    S(\lambda,t)=\begin{pmatrix}
                (1-h)S_0 + h S_1    & q \\
                q^T     & t
        \end{pmatrix},
\]
with $q=(\lambda_1,\lambda_2-\lambda_1,\dots,\lambda_N-\lambda_{N-1},1-\lambda_N)^T$ and where
the matrices $S_0, S_1\in \mathbb{S}^{N+1}$ are  defined by: {\small
\begin{equation}\label{b0matrix}
    S_0=
    \begin{pmatrix}
        2 \lambda_1     & -\lambda_1            \\
        -\lambda_1          & 2 \lambda_2   & -\lambda_2                \\
            &       -\lambda_2          & 2 \lambda_3   & -\lambda_3                \\
                                    &        & \ddots        &\ddots &   \ddots\\
                                            & &                 &       -\lambda_{N-1}  & 2 \lambda_N & -\lambda_N \\
                                            & &                 &        &  -\lambda_N  & 1\\
    \end{pmatrix}
\end{equation}
}
 and
{\small {\begin{equation}\label{b1matrix}
    S_1=
        \begin{pmatrix}
        2 \lambda_1         & \lambda_2-\lambda_1           &\dots          &\lambda_N-\lambda_{N-1}    &1-\lambda_N  \\
        \lambda_2-\lambda_1         & 2 \lambda_2                   &           &\lambda_N-\lambda_{N-1}    &1-\lambda_N \\

        \vdots & & \ddots & &\vdots \\
        \lambda_N-\lambda_{N-1}     & \lambda_N-\lambda_{N-1}   &       & 2\lambda_N    & 1-\lambda_N  \\
        1-\lambda_N         & 1-\lambda_N       & \dots         & 1-\lambda_N   & 1 \\
    \end{pmatrix}
    .
\end{equation}}}
\end{lemma}
\begin{proof}
For convenience,  we recast \eqref{P:gpr} as a minimization problem, and we also omit the
fixed term $LR^2$ from the objective. That is, we consider the equivalent problem \eqref{P:agpr}
defined by
\begin{align*}
    \begin{aligned}
    \min_{G\in \mathbb{R}^{(N+1)\times d},\delta \in\mathbb{R}^{N+1}}\ &  -\delta_N \\
    \text{s.t. }
                        & \trace (G^T A_{i-1,i} G) \leq \delta_{i-1} - \delta_i,\quad    i=1,\dots,N, \\
                        & \trace(G^T D_i G+ \nu u_i^T G ) \leq - \delta_i,\quad     i=0,\dots,N.
    \end{aligned}
    \tag{G$^{\prime\prime}$}\label{P:agpr}
\end{align*}

Attaching the dual multipliers $\lambda=(\lambda_1,\dots,\lambda_N) \in \real^N_{+}$ and
$\tau:=(\tau_0, \ldots ,\tau_N)^T \in \real^{N+1}_{+}$ to the first and second set of
inequalities respectively, and using the notation $\delta=(\delta_0, \ldots ,\delta_N)$, we
get that the Lagrangian of this problem is given as a sum of two separable functions in the
variables $(\delta, G)$:
\begin{eqnarray*}
    L(G,\delta, \lambda, \tau)&=&
         -\delta_N+\sum_{i=1}^{N} \lambda_i(\delta_i - \delta_{i-1}) + \sum_{i=0}^N \tau_i \delta_i \\
         & &+ \sum_{i=1}^{N} \lambda_i \trace(G^T A_{i-1,i} G)+ \sum_{i=0}^N \tau_i
        \left( \trace(G^T D_i G+ \nu u_i^T G) \right)\\
        & \equiv & L_1(\delta, \lambda, \tau)+ L_2(G,\lambda, \tau).
\end{eqnarray*}
The dual objective function is then defined by $$H(\lambda, \tau)=\min_{G, \delta} L(G,\delta,
\lambda\, \tau)=\min_\delta L_1(\delta, \lambda, \tau)+ \min_{G} L_2(G, \lambda, \tau),$$ and
the dual problem of \eqref{P:agpr} is then given by
\begin{align*}
    \begin{aligned}
        \max \{ H(\lambda, \tau):\; \lambda \in\real^{N}_{+}, \tau \in \real^{N+1}_{+}\}.
    \end{aligned}
    \tag{DG$^{\prime\prime}$}\label{P:dualagpr}
\end{align*}
Since $L_1(\cdot,\lambda, \tau)$ is linear in $\delta$, we have $\min_\delta
L_1(\delta,\lambda, \tau) =0$ whenever
\begin{eqnarray}
            -\lambda_1 +\tau_0 &= & 0, \nonumber\\
            \lambda_i-\lambda_{i+1}+\tau_i &=& 0\quad (i=1,\dots,N-1), \label{eqtau}\\
             -1+\lambda_N+\tau_N &=& 0, \nonumber
            \end{eqnarray}
            and $-\infty$ otherwise. Invoking Lemma~\ref{L:quadlemma}, we get
\begin{align*}
     \min_{G\in\mathbb{R}^{(N+1)\times d}} L_2(G,\lambda, \tau) =
     \min_{w\in\mathbb{R}^{N+1}} L_2(w \nu^T,\lambda, \tau).
\end{align*}
Therefore for any $(\lambda,\tau)$ satisfying (\ref{eqtau}), we have obtained that the dual
objective reduces to
\begin{eqnarray*}
 H(\lambda, \tau) &=& \min_{w \in \real^{N+1}}\{w^T \left(\sum_{i=1}^N \lambda_i  A_{i-1,i}  +
        \sum_{i=0}^N \tau_i  D_i \right) w + \tau^T w \}\\
        &=&  \max_{t\in \mathbb{R}} \{-\frac{1}{2}t :
        w^T \left(\sum_{i=1}^N \lambda_i  A_{i-1,i}  +
        \sum_{i=0}^N \tau_i  D_i \right) w + \tau^T w \leq -\frac{1}{2}t,\ \forall w\in \mathbb{R}^{N+1}
        \}\\
 &=& \max_{t\in \mathbb{R}} \left\{
            -\frac{1}{2}t:
            \begin{pmatrix}
                \sum_{i=1}^N \lambda_i  A_{i-1,i}  + \sum_{i=0}^N \tau_i  D_i  & \frac{1}{2} \tau \\
                \frac{1}{2}\tau^T       & \frac{1}{2}t
        \end{pmatrix} \succeq 0
        \right\}.
\end{eqnarray*}
where the last equality follows from  the well known lemma
\cite[Page~163]{ben2001lectures}\footnote{Let $M$ be a symmetric matrix. Then, $x^TMx + 2 b^Tx+ c \geq 0, \forall x \in
\real^d$ if and only if the matrix $\begin{pmatrix} M & b\\ b^T & c \end{pmatrix}$ is positive
semidefinite.}.

Now, recalling the definition of the matrices $A_{i-1,i}, D_i$ (see
(\ref{E:gradientmatrices})), we obtain {\small
\[
    \sum_{i=1}^N \lambda_i  A_{i-1,i}  =
    \frac{1}{2}\begin{pmatrix}
        \lambda_1       & (h-1)\lambda_1            \\
        (h-1)\lambda_1          & \lambda_1+\lambda_2   & (h-1)\lambda_2                \\
            &       (h-1)\lambda_2          & \lambda_2 +\lambda_3  & (h-1)\lambda_3                \\
                                    &        & \ddots        &\ddots &   \ddots\\
                                            & &                 &       (h-1)\lambda_{N-1}  & \lambda_{N-1}+ \lambda_N & (h-1)\lambda_N \\
                                            & &                 &        &  (h-1)\lambda_N  & \lambda_N\\
    \end{pmatrix}
    \]
    }
and {\small
\[
    \sum_{i=0}^N \tau_i  D_i =
    \frac{1}{2} \begin{pmatrix}
        \tau_0          & h\tau_1           & \dots     &h\tau_{N-1}    &h\tau_N  \\
        h\tau_1             & \tau_1        &               &h\tau_{N-1}    &h\tau_N \\

        \vdots  & & \ddots &     &\vdots \\
        h\tau_{N-1}     & h\tau_{N-1}   &           & \tau_{N-1}        & h\tau_N \\
        h\tau_N         & h\tau_N & \dots       & h\tau_N   & \tau_N \\
    \end{pmatrix}.
\]
}
Finally,  using the relations (\ref{eqtau}) to eliminate $\tau_i$, and recalling that
$\val\!\eqref{P:agpr}$ was defined as $-LR^2 \val\!\eqref{P:gpr}$, the desired form of the stated dual problem
follows.
\end{proof}

The next lemma will be crucial in invoking duality in the forthcoming theorem.
The proof for this lemma is quite technical and appears in the appendix.
\begin{lemma} \label{L:positivedefinitea0a1}
Let
\begin{align*}
    \lambda_i &= \frac{i}{2 N+1-i},\qquad i=1,\dots,N,
\end{align*}
then the matrices $S_0, S_1 \in \mathbb{S}^{N+1}$ defined in (\ref{b0matrix})--(\ref{b1matrix})
are positive definite for every $N\in\mathbb{N}$.
\end{lemma}

We are now ready to establish a new upper bound on the complexity of the gradient method for
values of $h$ between 0 and 1. To the best of our knowledge, the tightest bound thus far is
given by \eqref{E:oldgradientbound}.
\begin{theorem}\label{T:gradUpperBound}
Let $f\in C^{1,1}_L(\mathbb{R}^d)$ and let $x_0,\dots,x_N\in \mathbb{R}^d$ be generated by
Algorithm~\gradalg with $0< h \leq 1$. Then
\begin{equation}\label{E:gmbound}
    f(x_N)-f(x_\ast) \leq \frac{LR^2}{4 N h+2}.
\end{equation}
\end{theorem}
\begin{proof} First note that both (G) and \eqref{P:gpr} are clearly feasible 
and $\val (G) \leq\val\eqref{P:gpr}$.
Invoking Lemma \ref{lem:dual}, by weak duality for the pair of primal-dual problems
\eqref{P:gpr} and \eqref{P:dualgpr}, we thus obtain that $\val\eqref{P:gpr}\leq \val\eqref{P:dualgpr}$
and hence:
\begin{equation}\label{weak}
    f(x_N)-f^* \leq \val (G) \leq \val \eqref{P:gpr} \leq \val \eqref{P:dualgpr}.
\end{equation}

Now consider the following point $(\lambda,t)$ for the dual problem \eqref{P:dualgpr}:
\begin{align*}
    \lambda_i &= \frac{i}{2 N+1-i},\quad i=1,\dots,N,  \\
    t   &= \frac{1}{2Nh+1}.
\end{align*}
Assuming that this point is \eqref{P:dualgpr}-feasible, it follows from (\ref{weak}) that
$$
f(x_N)-f^* \leq \val\eqref{P:dualgpr} \leq \frac{L R^2}{4 N h+2},
$$
which proves the desired result. Thus, it remains to show that the above given choice
$(\lambda, t)$ is feasible for \eqref{P:dualgpr}. First, it is easy to see that all the linear constraints
of \eqref{P:dualgpr} on the variables $\lambda_i, i=1,\ldots ,N$ described through the set $\Lambda$
hold true. Now we prove that the matrix $S\equiv S(\lambda,t)$ is positive semidefinite. From
Lemma~\ref{L:positivedefinitea0a1}, with $h \in [0,1]$, we get that $(1-h)S_0 + h S_1$ is
positive definite, as a convex combination of positive definite matrices. Next, we argue that
the determinant of $S$ is zero. Indeed, take $u:=(1,\dots,1,-(2Nh+1))^T$, then from the
definition of $S$ and the choice of $\lambda_i$ and $t$ it follows by elementary algebra that
$Su=0$. To complete the argument, we note that the determinant of $S$ can also be found via
the identity (see, e.g., \cite[Section~A.5.5]{boyd2004}):
\[
    \det(S) = (t-q^T((1-h)S_0 + h S_1)^{-1}q) \det((1-h)S_0 + h S_1).
\]
Since we have just shown that  $(1-h)S_0 + h S_1 \succ 0$, then $\det((1-h)S_0 + h S_1)>0$ and
we get from the above identity that the value of $t-q^T((1-h)S_0 + h S_1)^{-1}q$, which is the
Schur complement of the matrix $S$, is equal to 0. By a well known lemma on the Schur
complement \cite[Lemma~4.2.1]{ben2001lectures}, we conclude that $S$ is positive semidefinite.
\end{proof}

The next theorem gives a lower bound on the complexity of Algorithm~\gradalg for all values of
$h$. In particular, it shows that the bound~\eqref{E:gmbound} is tight and that it can be
attained by a specific convex function in $C^{1,1}_L$.
\begin{theorem}\label{T:gradLowBound}
Let $L>0$, $N\in \mathbb{N}$ and $d\in \mathbb{N}$. Then for every $h>0$ there exists a convex
function $\varphi \in C^{1,1}_L(\mathbb{R}^d)$ and a point $x_0 \in \mathbb{R}^d$ such that
after $N$ iterations, Algorithm~\gradalg
reaches an approximate solution $x_N$ with following absolute inaccuracy
\[
    \varphi(x_N)-\varphi^\ast = \frac{L R^2}{2}
    \max \left( \frac{1}{2 N h+1},(1-h)^{2N}\right).
\]
\end{theorem}
\begin{proof}
We will describe two functions that attain the two parts of the $\max$ expression in the above
claimed statement. For the sake of simplicity we will
assume that 
$L=1$ and $R=\|x_\ast-x_0\|=1$.
Generalizing this proof to general values of $L$ and $R$
can be done by an appropriate scaling.

To show the first part of the $\max$ expression, consider the function
\[
    \varphi_1(x) =
    \begin{cases}
        \frac{1}{2Nh+1} \|x\| -\frac{1}{2(2Nh+1)^2}       &   \| x \| \geq \frac{1}{2Nh+1} \\
        \frac{1}{2} \|x\|^2         &   \| x\| < \frac{1}{2Nh+1}
    \end{cases}.
\]
Note that this function is nothing else but the Moreau proximal envelope of the function
$\frac{1}{2Nh+1}\|x\|$, \cite{M65}. It is well known that this function is convex, continuously
differentiable with Lipschitz constant $L=1$, and that its minimal value $\varphi(x_*)=0$, see
e.g., \cite{M65, rock-wets-book98}. Applying the gradient method on $\varphi_1(x)$ with
$x_0=\nu$ where, as before, $\nu$ is a unit vector in $\mathbb{R}^d$, we obtain that for $i=0,\dots,N$:
\begin{eqnarray*}
  \varphi_1^\prime(x_i) &=& \frac{1}{2N h+1}\nu;\;  x_i = \left(1-\frac{i h}{2N h+1}\right)\nu,\\
 \mbox{and}\; \varphi_1(x_i) &=&  \frac{1}{2N h+1}\left(1-\frac{h i}{2N h+1} \right)-\frac{1}{2(2N h+1)^2}\\
 &=& \frac{1}{4Nh+2} \left(\frac{4Nh+1-2hi}{2Nh+1}\right).
 \end{eqnarray*}
Therefore,
\[
    \varphi_1(x_N)-\varphi_1(x_\ast)  = \varphi_1(x_N) = \frac{1}{4 N h+2}.
\]
To show the second part of the $\max$ expression, we apply the gradient method on
\[
    \varphi_2(x) = \frac{1}{2} \|x\|^2
\]
with $x_0=\nu$. We then get that for $i=0,\dots,N$:
$$
    \varphi_2^\prime(x_i)= x_i;\;  x_i =(1-h)^i \nu;\; \varphi_2(x_i) = \frac{1}{2}(1-h)^{2i},
$$
and hence
\[
    \varphi_2(x_N)-\varphi_2(x_\ast) = \varphi_2(x_N) = \frac{1}{2} (1-h)^{2N}
\]
and the desired claim is proven.
\end{proof}

Numerical experiments we have performed on problem \eqref{P:gbound} suggest that the
lower complexity bound given by Theorem~\ref{T:gradLowBound} is in fact the exact complexity bound on the gradient method
with $0<h<2$.
\begin{conjecture}\label{C:gradtightbound}
Suppose the sequence $x_0,\dots,x_N$ is generated by the gradient method \gradalg with $0 <h < 2$,
then
\[
    f(x_N)-f(x_\ast)\leq \frac{L R^2}{2}
    \max \left( \frac{1}{2 N h+1},(1-h)^{2N}\right).
\]
\end{conjecture}
We now turn our attention to the problem of choosing the step size, $h$.
Assuming the above conjecture holds true,
the optimal step size (i.e., the step size that produces the best complexity bound)
for the gradient method with constant step size
is given by the unique positive solution to the equation $$\frac{1}{2 N
h+1}=(1-h)^{2N}.$$
The solution of this equation approaches $2$ as $N$ grows. Therefore, assuming the conjecture
holds true and $N$ is large enough, the complexity of the gradient method with the optimal
step size approaches~$\frac{LR^2}{8 N+2}$. This represents a significant improvement, by the
factor of 4, upon the best known bound on the gradient method \eqref{E:oldgradientbound}, and
also supports the observation that the gradient method often performs better in practice than
in theory.

\section{A Class of First-Order Methods: Numerical Bounds}\label{S:numerical_bounds}
The framework developed in the previous sections 
will now serve as a basis to extend the performance analysis for  a broader class of first-order methods
for minimizing a smooth convex function over $\mathbb{R}^d$. First, we define a general class
of first-order algorithms (FO) and we show that it encompasses some interesting first-order methods.
Then, following our approach, we define the corresponding PEP associated with the class of
algorithms (FO). Although for this more general case, an analytical solution is not available
for determining the bound, we establish that given a fixed number of steps $N$, a
{\em numerical bound} on the performance of (FO) can be efficiently computed. We then illustrate how this result can
be applied for deriving new complexity bounds on two first-order methods, which belong to the class (FO).

\subsection{A General First-Order Algorithm: Definition and Examples}

As before, our family $\mathcal{F}$ is the class of convex functions in $C^{1,1}_L(\real^d)$,
and $\{d,N,L, R\}$ are fixed. Consider the following class of first-order methods:
\medskip

\fbox{\parbox{14cm}{ {\bf Algorithm FO} \medskip
\begin{enumerate}
\setcounter{enumi}{-1}
\item Input: $f\in C^{1,1}_L(\mathbb{R}^d)$, $x_0\in \mathbb{R}^d.$
\item For $i=0,\dots,N-1$, compute $x_{i+1} = x_i - \frac{1}{L} \sum_{k=0}^{i} h^{(i+1)}_{k} f^\prime(x_k)$.
\end{enumerate}
}}
\medskip

Here, $h_k^{(i)} \in \real$ play the role of step-sizes, which we assume to be fixed and determined by
the algorithm.

The interest in the analysis of first-order algorithms of this type is motivated by the fact
that it covers some fundamental first-order schemes beyond the gradient method. In particular,
to motivate (FO), let us consider the following two algorithms which are of particular
interest, and as we shall see below can be seen as special cases of (FO).

We start with the so-called Heavy Ball Method (HBM). For earlier work on this method see
Polyak \cite{polyak1964}, and for some interesting modern developments and applications, we
refer the reader to Attouch et al. \cite{atto-goud-redo-97, atto-bolt-redo-02} and references
therein.

\begin{example}
\textbf{The Heavy Ball Method (HBM)}

\fbox{\parbox{14cm}{ {\bf Algorithm HBM} \medskip
\begin{enumerate}
\setcounter{enumi}{-1}
\item Input: $f\in C^{1,1}_L(\mathbb{R}^d)$, $x_0\in \mathbb{R}^d$,
\item $x_1 \leftarrow x_0 - \frac{\alpha}{L} f^\prime(x_0)$
\item For $i=1,\dots,N-1$ compute: $x_{i+1} = x_i - \frac{\alpha}{L} f^\prime(x_i) +\beta (x_i-x_{i-1})$
 \label{S:heavyballstep}
\end{enumerate}
}}
\medskip

Here the step sizes $\alpha$ and  $\beta$ are chosen such that $0 \leq \beta < 1$ and $0 < \alpha <
2(1+\beta)$, see \cite{polyak1964}.

By recursively eliminating the term $x_i-x_{i-1}$ in step 2 of HBM, we can rewrite this step
as
\[
    x_{i+1}= x_i - \frac{1}{L} \sum_{k=0}^{i} \alpha \beta^{i-k} f^\prime(x_k),\quad i=1,\dots,N-1.
\]
Therefore, the heavy ball method is a special case of Algorithm~\mainalg with the choice
$$h^{(i+1)}_{k}= \alpha \beta^{i-k}, \quad  k=0, \ldots, i, \; i=0, \ldots N-1 .$$
\end{example}

\medskip

The next algorithm  is Nesterov's celebrated \emph{Fast Gradient Method} \cite{nest83}.

\begin{example}\textbf{The fast gradient method (FGM)}

\fbox{\parbox{14cm}{ {\bf Algorithm FGM} \medskip
\begin{enumerate}
\setcounter{enumi}{-1}
\item Input: $f\in C^{1,1}_L(\mathbb{R}^d)$, $x_0\in \mathbb{R}^d$,
\item $y_1 \leftarrow x_0$, $t_1 \leftarrow 1$,
\item For $i=1,\dots,N$ compute:
\begin{enumerate}
    \item $x_i \leftarrow y_i -\frac{1}{L} f^\prime(y_i)$, \label{S:xdef}
    \item $t_{i+1} \leftarrow \frac{1+\sqrt{1+4 t_i^2}}{2}$,
    \item $y_{i+1} \leftarrow x_i + \frac{t_i-1}{t_{i+1}}(x_i-x_{i-1})$. \label{S:ydef}
\end{enumerate}
\end{enumerate}
}}
\end{example}

A major breakthrough was achieved by Nesterov in \cite{nest83}, where he proved that the FGM,
which requires almost no increase in computational effort when compared to the basic gradient
scheme, achieves the improved rate  of convergence $O(1/N^2)$ for function values. 
More precisely, one has\footnote{The bound given here, which improves the original bound derived in \cite{nest83}, was recently obtained in Beck-Teboulle \cite{beck-tebo-fista09}.}
\begin{equation}\label{E:nesterovBound}
    f(x_N)-f^\ast \leq \frac{2L\|x_0-x_\ast\|^2}{(N+1)^2}, \quad \forall x_\ast \in X_\ast(f).
\end{equation}
The order of complexity of Nesterov's algorithm is also {\em optimal}, as it is possible to
show that there exists a convex function $f\in C^{1,1}_L(\mathbb{R}^d)$ such that when $d\geq 2N+1$,
and under some other mild assumptions, \emph{any} first-order algorithm that generates a point
$x_N$ by performing $N$ calls to a first-order oracle of $f$ satisfies
\cite[Theorem~2.1.7]{nest-book-04}
\[
    f(x_N)-f^\ast \geq \frac{3L \|x_0-x_\ast\|^2}{32(N+1)^2}, \quad \forall x_\ast \in X_\ast(f).
\]
This fundamental algorithm discovered about 30 years ago by Nesterov \cite{nest83} has been
recently revived and is currently subject of intensive research activities. For some of its
extensions and many applications, see e.g., the recent survey paper Beck-Teboulle
\cite{beck-tebo-survey10} and references therein.
\medskip

At first glance, Algorithm~FGM seems different than the scheme \mainalg defined above.
Here two sequences are defined: the main sequence $x_0,\dots,x_N$ and an auxiliary sequence $y_1,\dots,y_N$.
Observing that the gradient of the function is only evaluated on the {\em auxiliary sequence}
of points $\{y_i\}$, we show in the next proposition that
FGM fits in the scheme \mainalg through the following algorithm:

\fbox{\parbox{14cm}{ {\bf Algorithm \fastgradalgb} \medskip
\begin{enumerate}
\setcounter{enumi}{-1}
\item Input: $f\in C^{1,1}_L(\mathbb{R}^d)$, $x_0\in \mathbb{R}^d$,
\item $y_1 \leftarrow x_0$,
\item For $i=1,\dots,N-1$ compute:
\begin{enumerate}
    \item$y_{i+1} \leftarrow y_i  -\frac{1}{L} \sum_{k=1}^{i} h^{(i+1)}_{k} f^\prime(y_k)$, \label{S:nesterovserialstep}
\end{enumerate}
\item $x_N \leftarrow y_N -\frac{1}{L} f^\prime(y_N)$,
\end{enumerate}
}}

with
\begin{align}
    h^{(i+1)}_k = \begin{cases}
        \frac{t_{i}-1}{t_{i+1}}h^{(i)}_k    & k+2\leq i, \\
        \frac{t_{i}-1}{t_{i+1}}(h^{(i)}_{i-1}-1)    & k = i-1, \\
        1+\frac{t_{i}-1}{t_{i+1}}   & k = i
    \end{cases} \qquad (i=1,\dots,N-1, \ k=1,\dots,i). \label{E:nesterovsteps}
\end{align}
and
\begin{align*}
    t_i =
        \begin{cases}
            1       & i=1 \\
            \frac{1+\sqrt{1+4 t_{i-1}^2}}{2} & i>1.
        \end{cases}
\end{align*}

\begin{proposition}
The points $y_1,\dots,y_N,x_N$ generated by Algorithm~\fastgradalgb are identical to the
respective points generated by Algorithm~\fastgradalg.
\end{proposition}
\begin{proof}
We will show by induction that the sequence $y_i$ generated by Algorithm~\fastgradalgb is
identical to the sequence $y_i$ generated by Algorithm~\fastgradalg, and that the value of
$x_N$ generated by Algorithm~\fastgradalgb is equal to the value of $x_N$ generated by
Algorithm~\fastgradalg.

First note that the sequence $t_i$ is defined by the two algorithms in the same way. Now let
$\{x_i$, $y_i\}$ be the sequences generated by \fastgradalg and denote by $\{y_i^\prime\}$,
$x_N^\prime$ the sequence generated by \fastgradalgb. Obviously, $y^\prime_1 = y_1$ and since
$t_1=1$ we get using the relations \ref{E:nesterovsteps}:
\begin{align*}
    y_2^\prime &= y_1^\prime - \frac{1}{L} h^{(2)}_1 f^\prime(y_1^\prime) = y_1 - \frac{1}{L} \left(1+\frac{t_{1}-1}{t_2} \right) f^\prime(y_1) = y_1 - \frac{1}{L} f^\prime(y_1)= x_1 = y_2.
\end{align*}
Assuming $y_i^\prime = y_i$ for $i = 1,\dots,n$, we then have
\begin{align*}
    y_{n+1}^\prime &= y_n^\prime - \frac{1}{L} h^{(n+1)}_n f^\prime(y_n^\prime) - \frac{1}{L} h^{(n+1)}_{n-1} f^\prime(y_{n-1}^\prime) - \frac{1}{L} \sum_{k=1}^{n-2} h^{(n+1)}_k f^\prime(y_k^\prime) \\
        &= y_n - \frac{1}{L} (1+\frac{t_n-1}{t_{n+1}}) f^\prime(y_n) - \frac{1}{L} \frac{t_n-1}{t_{n+1}}(h^{(n)}_{n-1}-1) f^\prime(y_{n-1}) - \frac{1}{L} \sum_{k=1}^{n-2} \frac{t_n-1}{t_{n+1}}h^{(n)}_k f^\prime(y_k^\prime) \\
        &= y_n - \frac{1}{L} f^\prime(y_n) +\frac{t_n-1}{t_{n+1}} \left( -\frac{1}{L}  f^\prime(y_n) + \frac{1}{L} f^\prime(y_{n-1}) - \frac{1}{L} \sum_{k=1}^{n-1} h^{(n)}_k f^\prime(y_k^\prime)\right) \\
        &= x_n +\frac{t_n-1}{t_{n+1}} \left( -\frac{1}{L}  f^\prime(y_n) + \frac{1}{L} f^\prime(y_{n-1}) +y_n^\prime-y_{n-1}^\prime\right) \\
        &= x_n +\frac{t_n-1}{t_{n+1}} (x_n-x_{n-1}) \\
        &= y_{n+1}.
\end{align*}
Finally,
\[
    x_N^\prime = y_N^\prime -\frac{1}{L} f^\prime(y_N^\prime) = y_N -\frac{1}{L} f^\prime(y_N) = x_N.
\]
\end{proof}

\subsection{A Numerical Bound for \mainalg}

To build the performance estimation problem for Algorithm~\mainalg, from which a complexity
bound can be derived, we follow the approach 
used to derive problem \eqref{P:gbound} for the gradient method. The only difference being
that here, of course, the relation between the variables $x_i$ is derived from the main
iteration of Algorithm~\mainalg. After some algebra, the resulting PEP for the class of
algorithms (FO) reads
\begin{align*}
    \begin{aligned}
    \max_{G\in \mathbb{R}^{(N+1)\times d},\delta_i \in\mathbb{R}}\ &  LR^2 \delta_N \\
    \text{s.t. }
                        & \trace (G^T \tilde{A}_{i,j} G) \leq \delta_i - \delta_j,\quad    i<j=0,\dots,N, \\
                        & \trace(G^T \tilde{B}_{i,j} G) \leq \delta_i - \delta_j,\quad      j<i=0,\dots,N, \\
                        & \trace(G^T \tilde{C}_i G) \leq \delta_i,\quad                             i=0,\dots,N, \\
                        & \trace(G^T \tilde{D}_i G+ \nu u_i^T G ) \leq - \delta_i,\quad     i=0,\dots,N,
    \end{aligned}
    \tag{Q}\label{P:fop}
\end{align*}
where  $\tilde{A}_{i,j}$, $\tilde{B}_{i,j}$, $\tilde{C}_i$ and $\tilde{D}_i$ are defined, similarly to \eqref{E:gradientmatrices}, by
\begin{equation}\label{E:defmatrices}
\begin{aligned}
    \tilde{A}_{i,j} &= \frac{1}{2} (u_i-u_j)(u_i-u_j)^T + \frac{1}{2}\sum_{t=i+1}^{j} \sum_{k=0}^{t-1} h_k^{(t)}(u_j u_k^T+u_k u_j^T),\\
    \tilde{B}_{i,j} &= \frac{1}{2} (u_i-u_j)(u_i-u_j)^T -  \frac{1}{2}\sum_{t=j+1}^{i} \sum_{k=0}^{t-1} h_k^{(t)}(u_j u_k^T+u_k u_j^T),\\
    \tilde{C}_i &= \frac{1}{2} u_i u_i^T,\\
    \tilde{D}_i &= \frac{1}{2} u_i u_i^T+  \frac{1}{2}\sum_{t=1}^{i} \sum_{k=0}^{t-1} h_k^{(t)} (u_i u_k^T + u_k u_i^T)
\end{aligned}
\end{equation}
and we recall that $\nu \in \mathbb{R}^d$ is a given unit vector, $u_i=e_{i+1}\in\mathbb{R}^{N+1}$ and
the notation $i <j=0,\dots,N$ is a shorthand notation for $i=0,\dots,N-1, j= i+1,\dots, N$.

In view of the difficulties in the analysis required to find the solution of \eqref{P:gbound},
an analytical solution to this more general case seems unlikely. However, as we now proceed to
show, we can derive a {\em numerical bound} on this problem that can be efficiently computed.

Following  the analysis of the gradient method, (cf.\eqref{P:gpr} in \S\ref{sub-tight}) we consider the
following simpler relaxed problem:
\begin{align*}
    \begin{aligned}
    \max_{G\in \mathbb{R}^{(N+1)\times d},\delta_i \in\mathbb{R}}\ &  LR^2 \delta_N \\
    \text{s.t. }
                        & \trace (G^T \tilde{A}_{i-1,i} G) \leq \delta_{i-1} - \delta_i,\quad    i=1,\dots,N, \\
                        & \trace(G^T \tilde{D}_i G+ \nu u_i^T G ) \leq - \delta_i,\quad     i=0,\dots,N.
    \end{aligned}
    \tag{Q$^\prime$}\label{P:fopr}
\end{align*}

With the same proof as given in Lemma~\ref{lem:dual}  for problem \eqref{P:fopr}, we obtain that a dual
problem for \eqref{P:fopr} is given by the following convex semidefinite optimization problem:

\begin{align*}
    \begin{aligned}
    \min_{\lambda, \tau,t}\ & \frac{1}{2} t \\
    \text{s.t. }
            & \begin{pmatrix}
                \sum_{i=1}^N \lambda_i  \tilde{A}_{i-1,i}  + \sum_{i=0}^N \tau_i  \tilde{D}_i  &\frac{1}{2} \tau \\
                \frac{1}{2} \tau^T      & \frac{1}{2} t
        \end{pmatrix} \succeq 0, \\
        & (\lambda, \tau) \in \tilde{\Lambda},
    \end{aligned}
    \tag{DQ$^\prime$}\label{P:foprd}
\end{align*}
where
\begin{equation}\label{E:duallinearconstraints}
    \tilde{\Lambda} = \{ (\lambda,\tau) \in \real^{N}_{+}\times \real^{N+1}_{+} :
         \tau_0 = \lambda_1, \
        \lambda_i-\lambda_{i+1}+\tau_i = 0, \  i=1,\dots,N-1,\
        \lambda_N+\tau_N=1
    \}.
\end{equation}

Note that the data matrices of both primal-dual problems \eqref{P:fopr}
and \eqref{P:foprd} depend on the step-sizes $h^{(i)}_k$.
To avoid a trivial bound for problem $\eqref{P:fopr}$, here we need the following assumption
on the dual problem \eqref{P:foprd}:

{\bf Assumption 1} Problem \eqref{P:foprd} is solvable, i.e., the minimum is finite and attained for the
given step-sizes $h^{(i)}_k$.

Actually, the attainment requirement can be avoided if we can exhibit a feasible point
$(\lambda, \tau, t)$ for the problem \eqref{P:foprd}. As noted earlier, given the difficulties already
encountered for the simpler gradient method, finding explicitly such a point for the general
class of algorithms (FO) is unlikely. However, the structure of problem  \eqref{P:foprd} will be very
helpful in the analysis of the next section which further addresses the role of the
step-sizes.

The promised numerical bound now easily follows showing that a complexity bound for (FO) is
determined by the optimal value of the dual problem \eqref{P:foprd} which can be computed
efficiently by any numerical algorithm for SDP (see e.g.,
\cite{ben2001lectures,Helmberg96aninterior-point,vandenberghe1996semidefinite}).

\begin{proposition}\label{numb} Fix any $N , d \in \mathbb{N}$. Let $f \in C^{1,1}_L(\real^d)$
be convex and suppose that $x_0, \ldots ,x_N \in \real^d$ are generated by Algorithm \mainalg,
and  that assumption 1 holds.
Then,
$$
f(x_N) - f^* \leq LR^2 \val\eqref{P:foprd}.
$$
\end{proposition}

\begin{proof} Follows from weak duality for the pair of primal-dual problems \eqref{P:fopr}-\eqref{P:foprd}
\end{proof}




\subsection{Numerical Illustrations}\label{s:numeril}
We apply Proposition \ref{numb} to find bounds on the complexity of the heavy ball method
(HBM) with\footnote{According to our simulations, this choice for the values of $\alpha, \beta$ produce
results that are typical of the behavior of the algorithm.}
$\alpha=1$ and $\beta=\frac{1}{2}$  and on the fast
gradient method (FGM) with $h_k^{(i)}$ as given in (\ref{E:nesterovsteps}), which as shown
earlier, can both be viewed as particular realizations of (FO).

The resulting SDP programs were solved for different values of $N$ using
\texttt{CVX}~\cite{cvx,gb08}.
 These results, together with the classical bound on the convergence rate of the main
sequence of the fast gradient method \eqref{E:nesterovBound}, are summarized in
Figures~\ref{fig:exampleNesterovExample} and \ref{F:newNeterovBound}.
\medskip

\begin{figure}[H]
  \centerline{\includegraphics[scale=0.5]{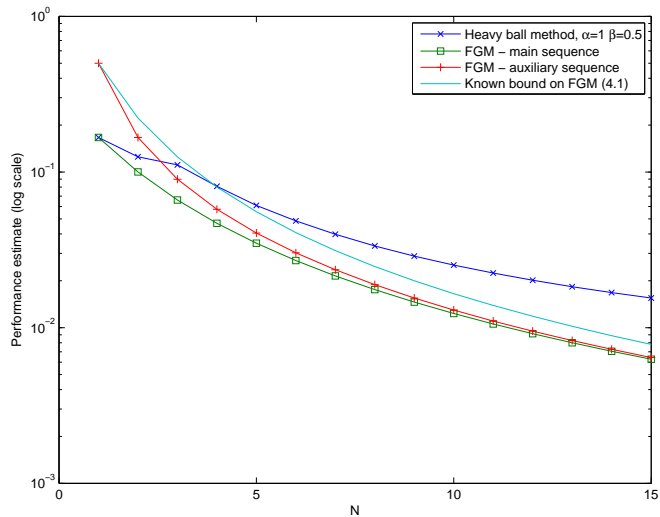}}
  \caption{The new bounds on the heavy ball and fast gradient methods.}
  \label{fig:exampleNesterovExample}
\end{figure}

\begin{figure}
\centerline{
\begin{tabular}{lllll}
N  &  Heavy Ball  &  FGM, main  & FGM, auxiliary  & Known bound on FGM \eqref{E:nesterovBound} \\
1 & LR$^2$/6.00 & LR$^2$/6.00 & LR$^2$/2.00 & LR$^2$/2.0=2LR$^2$/(1+1)$^2$\\
2 & LR$^2$/7.99 & LR$^2$/10.00 & LR$^2$/6.00 & LR$^2$/4.5=2LR$^2$/(2+1)$^2$\\
3 & LR$^2$/9.00 & LR$^2$/15.13 & LR$^2$/11.13 & LR$^2$/8.0=2LR$^2$/(3+1)$^2$\\
4 & LR$^2$/12.35 & LR$^2$/21.35 & LR$^2$/17.35 & LR$^2$/12.5=2LR$^2$/(4+1)$^2$\\
5 & LR$^2$/16.41 & LR$^2$/28.66 & LR$^2$/24.66 & LR$^2$/18.0=2LR$^2$/(5+1)$^2$\\
10 & LR$^2$/39.63 & LR$^2$/81.07 & LR$^2$/77.07 & LR$^2$/60.5=2LR$^2$/(10+1)$^2$\\
20 & LR$^2$/89.45 & LR$^2$/263.65 & LR$^2$/259.65 & LR$^2$/220.5=2LR$^2$/(20+1)$^2$\\
40 & LR$^2$/188.99 & LR$^2$/934.89 & LR$^2$/930.89 & LR$^2$/840.5=2LR$^2$/(40+1)$^2$\\
80 & LR$^2$/387.91 & LR$^2$/3490.22 & LR$^2$/3486.22 & LR$^2$/3280.5=2LR$^2$/(80+1)$^2$\\
160 & LR$^2$/785.68 & LR$^2$/13427.43 & LR$^2$/13423.43 & LR$^2$/12960.5=2LR$^2$/(160+1)$^2$\\
500 & LR$^2$/2476.11 & LR$^2$/127224.44 & LR$^2$/127220.32 & LR$^2$/125500.5=2LR$^2$/(500+1)$^2$\\
1000 & LR$^2$/4962.01 & LR$^2$/504796.99 & LR$^2$/504798.28 &
LR$^2$/501000.5=2LR$^2$/(1000+1)$^2$
\end{tabular}}
\caption{The new bounds for HBM and FGM versus the classical bound on Nesterov's algorithm.}
\label{F:newNeterovBound}
\end{figure}

Note that as far as the authors are aware, there is no known convergence rate result for the
HBM on the class of convex functions in $C^{1,1}_L$. As can be seen from the above results,
the numerical bound for HBM behaves slightly better than the gradient method (compare with the
explicit bound given in Theorem \ref{T:gradUpperBound}), but remains much slower than the fast
gradient scheme (FGM).

Considering the results on the FGM,
note that the numerical bounds for the main sequence of point $x_i$ and the corresponding
values at the auxiliary sequence $y_i$ of the fast gradient method are very similar and
perform slightly better than predicted by the classical bound \eqref{E:nesterovBound}.
To the best of our knowledge, the complexity of the auxiliary sequence is yet unknown,
thus these results encourage us to raise the following conjecture.
\begin{conjecture}\label{conj:xy}
Let $x_0,x_1,\dots$ and $y_1,y_2,\dots$ be the main and auxiliary sequences defined by FGM
(respectively), then $\{f(x_i)\}$ and $\{f(y_i)\}$ converge to the optimal value of the problem
with the same rate of convergence.
\end{conjecture}

\section{A Best Performing Algorithm: Optimal Step Sizes for The Algorithm Class \mainalg}\label{S:optalg}

We now consider the problem of finding the ``best" performing algorithm of the form \mainalg
with respect to the new bounds. Namely, we consider the problem of minimizing
$\val\eqref{P:fopr}$, the optimal value of \eqref{P:fopr},
with respect to the step sizes $h:=(h^{(i)}_k)_{0\leq k<i \leq N}$
defining the algorithm \mainalg, and which are now considered as unknown variables in
\mainalg.


We denote by  $\tilde{A}_{i,j}(h)$ and $\tilde{D}_i(h)$, the matrices given in
\eqref{E:defmatrices},  which are functions of the algorithm step sizes $h$. The resulting
bound derived in Proposition \ref{numb} is thus a function of $h$, and the problem of
minimizing $\val \eqref{P:foprd}$ with respect to the step sizes $h$ thus consists of solving
the following bilinear problem: {\small
\begin{align*}
    \begin{aligned}
    \min_{h,\lambda, \tau,t}\ \left \{ \frac{1}{2} t :
           \begin{pmatrix}
                \sum_{i=1}^N \lambda_i  \tilde{A}_{i-1,i}(h)  + \sum_{i=0}^N \tau_i  \tilde{D}_i(h)  &\frac{1}{2} \tau \\
                \frac{1}{2} \tau^T      & \frac{1}{2} t
        \end{pmatrix} \succeq 0,
         (\lambda, \tau) \in \tilde{\Lambda} \right\},
    \end{aligned}
    \tag{BIL}\label{P:pblin}
\end{align*}
} with $\tilde{\Lambda}$ defined as in \eqref{E:duallinearconstraints}.

Note that
the feasibility of \eqref{P:pblin} follows from
the proof of Theorem~\ref{T:gradLowBound}, where an explicit feasible point is given to \eqref{P:dualgpr},
which is a special instance of \eqref{P:pblin} when the steps $(h^{(i)}_k)$ are chosen as in the gradient method.

From the definition of the matrices $\tilde{A}_{i,j}(h)$ and $\tilde{D}_i(h)$, we get
\begin{align*}
    \sum_{i=1}^N \lambda_i  \tilde{A}_{i-1,i}(h)  + \sum_{i=0}^N \tau_i  \tilde{D}_i(h)
    &= \frac{1}{2}\sum_{i=1}^{N} \lambda_i (u_{i-1}-u_i)(u_{i-1}-u_i)^T +\frac{1}{2}\sum_{i=0}^{N}
    \tau_i u_i u_i^T  \\&\quad+  \frac{1}{2}\sum_{i=1}^{N} \sum_{k=0}^{i-1}
    \left( \lambda_i h_k^{(i)}+\tau_i \sum_{t=k+1}^{i} h_k^{(t)}\right) (u_i u_k^T + u_k
    u_i^T).
\end{align*}
Introducing the new variables:
\begin{equation}\label{E:definex}
    r_{i,k}=  \lambda_i h_k^{(i)}+\tau_i \sum_{t=k+1}^{i} h_k^{(t)}, \quad i=1,\dots,N,\ k=0,\dots,i-1
\end{equation}
and denoting $r=(r_{i,k})_{0\leq k<i \leq N}$, we obtain the following \emph{linear} SDP
relaxation of \eqref{P:pblin}:
\begin{align*}
    \begin{aligned}
    \min_{r,\lambda, \tau,t}\ \left\{ \frac{1}{2} t :
            \begin{pmatrix}
                S(r,\lambda,\tau)  &\frac{1}{2} \tau \\
                \frac{1}{2} \tau^T      & \frac{1}{2} t
        \end{pmatrix} \succeq 0, \
        (\lambda, \tau) \in \tilde{\Lambda} \right\},
    \end{aligned}
    \tag{LIN}\label{P:plin}
\end{align*}
where
\begin{align*}
     S(r,\lambda,\tau)
    &= \frac{1}{2}\sum_{i=1}^{N} \lambda_i (u_{i-1}-u_i)(u_{i-1}-u_i)^T +\frac{1}{2}\sum_{i=0}^{N} \tau_i u_i u_i^T  +  \frac{1}{2}\sum_{i=1}^{N} \sum_{k=0}^{i-1} r_{i,k} (u_i u_k^T + u_k u_i^T).
\end{align*}

This convex SDP can now be efficiently solved by numerical methods. As the following theorem
shows, its solution can be used to construct a solution for \eqref{P:pblin} with optimal step
sizes $h$.
\begin{theorem}
Suppose $(r^\ast,\lambda^\ast,\tau^\ast,t^\ast)$ is an optimal solution for (LIN),
then $(h,\lambda^\ast,\tau^\ast,t^\ast)$ is an optimal solution for \eqref{P:pblin}, where
$h=(h_k^{(i)})_{0\leq k<i \leq N}$ is defined by the following recursive rule
\begin{equation}\label{E:solvebil}
    h_k^{(i)} = \begin{cases}
            \frac{\tau_i^\ast \sum_{t=k+1}^{i-1} h_k^{(t)} -r_{i,k}^\ast }{\lambda_{i}^\ast} & \lambda_i^\ast \neq 0\\
            0   & \lambda_i^\ast = 0
        \end{cases},  \quad i=1,\dots,N,\ k=0,\dots,i-1.
\end{equation}
\end{theorem}
\begin{proof}
As (LIN) is a relaxation of \eqref{P:pblin}, it is enough to show that \eqref{P:pblin} can achieve the same
objective value. Let $(r^\ast,\lambda^\ast,\tau^\ast,t^\ast)$ be an optimal solution
for (LIN). If $\lambda_i^\ast \neq 0$ for all $1\leq i \leq N$, then~\eqref{E:solvebil}
satisfies all the equations in \eqref{E:definex} and therefore
$(h,\lambda^\ast,\tau^\ast,t^\ast)$ is feasible for \eqref{P:pblin}.

Suppose $\lambda^\ast_m=0$ for some $1\leq m \leq N$ and that $m$ is the maximal index with
this property. Then by the equality and non-negativity constraints in (LIN), we get that
$\lambda^\ast_1=\lambda^\ast_2=\dots=\lambda^\ast_m=0$ and
$\tau^\ast_0=\tau^\ast_1=\dots=\tau^\ast_{m-1}=0$.
Let $S:=S(r,\lambda^\ast,\tau^\ast)$, then by the positive semidefinite constraint in (LIN), we have $S \succeq 0$.
From the linear equalities connecting $\lambda$ and $\tau$ it follows that
\[
    S_{i,i}=\begin{cases}
            \frac{1}{2} (\lambda_1^\ast+\tau_0^\ast) = \lambda_1^\ast, &\quad i=1 \\
            \frac{1}{2} (\lambda_i^\ast+\lambda_{i-1}^\ast+\tau_{i-1}^\ast) = \lambda_i^\ast, &\quad i=2,\dots,N
        \end{cases},
\]
and we get that $S_{1,1}=\dots=S_{m,m}=0$. By the properties of positive semidefinite matrices we now get that $r_{i,k}^\ast=0$ for $i=1,\dots,m$ and $k=0,\dots,i-1$,
hence the set of equations \eqref{E:definex} with the chosen values of $h_k^{(i)}$ is consistent.
\end{proof}

The optimal value of (LIN) for various values of $N$ is summarized in Figure~\ref{F:optbound}.
The resulting new algorithm with the computed optimal step sizes $h_k^{(i)}$ is illustrated
for $N=5$ and given in Figure~\ref{F:optalg5}. As can be seen from these results, (compare
with Figure \ref{F:newNeterovBound}) the performance of the new algorithm is almost exactly
two times better than the performance of the fast gradient method.

\begin{figure}[h!]
\centerline{
\begin{tabular}{lll}
N  &  val(LIN) \\
1 & LR$^2$/8.00 \\
2 & LR$^2$/16.16 \\
3 & LR$^2$/26.53 \\
4 & LR$^2$/39.09 \\
5 & LR$^2$/53.80 \\
10 & LR$^2$/159.07 \\
20 & LR$^2$/525.09 \\
40 & LR$^2$/1869.22 \\
80 & LR$^2$/6983.13 \\
160 & LR$^2$/26864.04 \\
500 & LR$^2$/254482.61 \\
1000 & LR$^2$/1009628.17
\end{tabular}}
\caption{The value of the optimal value of (LIN) for various values of $N$.} \label{F:optbound}
\end{figure}

\begin{figure}[h!]
\begin{align*}
& x_{1}\leftarrow x_{0}+ \textstyle\frac{1.6180}{L} f^\prime(x_{0})\\
& x_{2}\leftarrow x_{1}+ \textstyle\frac{0.1741}{L} f^\prime(x_{0})+ \textstyle\frac{2.0194}{L} f^\prime(x_{1})\\
& x_{3}\leftarrow x_{2}+ \textstyle\frac{0.0756}{L} f^\prime(x_{0})+ \textstyle\frac{0.4425}{L} f^\prime(x_{1})+ \textstyle\frac{2.2317}{L} f^\prime(x_{2})\\
& x_{4}\leftarrow x_{3}+ \textstyle\frac{0.0401}{L} f^\prime(x_{0})+ \textstyle\frac{0.2350}{L} f^\prime(x_{1})+ \textstyle\frac{0.6541}{L} f^\prime(x_{2})+ \textstyle\frac{2.3656}{L} f^\prime(x_{3})\\
& x_{5}\leftarrow x_{4}+ \textstyle\frac{0.0178}{L} f^\prime(x_{0})+ \textstyle\frac{0.1040}{L} f^\prime(x_{1})+ \textstyle\frac{0.2894}{L} f^\prime(x_{2})+ \textstyle\frac{0.6043}{L} f^\prime(x_{3})+ \textstyle\frac{2.0778}{L} f^\prime(x_{4})
\end{align*}
\caption{A first-order  algorithm  with optimal step-sizes for $N=5$.} \label{F:optalg5}
\end{figure}

\section{Acknowledgements}
This work was initiated during our participation to the ``Modern Trends in Optimization and
Its Application'' program at IPAM, (UCLA), September-December 2010. We would like to thank
IPAM for their support and for the very pleasant and stimulating environment provided to us
during our stay. We would also like to thank Simi Haber, Ido Ben-Eliezer and Rani Hod for
their help in the proof of Lemma~\ref{L:positivedefinitea0a1}.

\appendix

\section{Proof of Lemma~\ref{L:positivedefinitea0a1}}
We now establish the positive definiteness of the matrices $S_0$ and $S_1$
given in \eqref{b0matrix} and \eqref{b1matrix}, respectively.

\subsection{$S_0 \succ 0$}
We begin by showing that $S_0$ is positive definite. Recall that
\[
    S_0=
    \begin{pmatrix}
        2 \lambda_1     & -\lambda_1            \\
        -\lambda_1          & 2 \lambda_2   & -\lambda_2                \\
            &       -\lambda_2          & 2 \lambda_3   & -\lambda_3                \\
                                    &        & \ddots        &\ddots &   \ddots\\
                                            & &                 &       -\lambda_{N-1}  & 2 \lambda_N & -\lambda_N \\
                                            & &                 &        &  -\lambda_N  & 1\\
    \end{pmatrix}
\]
for
\begin{align*}
    \lambda_i &= \frac{i}{2 N+1-i}, \qquad i=1,\dots,N.
\end{align*}

Let us look at $x^T S_0 x$ for any $x=(x_0,\dots,x_N)^T$:
\begin{align*}
    x^T S_0 x
        &= \sum_{i=0}^{N-1} 2\lambda_{i+1} x_i^2 -2\sum_{i=0}^{N-1} \lambda_{i+1} x_i x_{i+1} + x_N^2\\
        &= \sum_{i=0}^{N-1} \lambda_{i+1} (x_{i+1}-x_i)^2 +\lambda_1 x_0^2+\sum_{i=1}^{N-1} (\lambda_{i+1}-\lambda_i) x_i^2+(1-\lambda_N)x_N^2
\end{align*}
which is always positive for $x\neq 0$. We conclude that $S_0$ is positive definite.

\subsection{$S_1 \succ 0$}

We will show that $S_1$ is positive definite using Sylvester's criterion\footnote{Despite
the interesting structure of the matrix $S_1$, this proof is quite involved. A simpler proof
would be most welcome!}.

Recall that
\[
    S_1=
    \begin{pmatrix}
        2 \lambda_1         & \lambda_2-\lambda_1           &\dots          &\lambda_N-\lambda_{N-1}    &1-\lambda_N  \\
        \lambda_2-\lambda_1         & 2 \lambda_2                   &          &\lambda_N-\lambda_{N-1}    &1-\lambda_N \\

        \vdots & & \ddots & &\vdots \\
        \lambda_N-\lambda_{N-1}     & \lambda_N-\lambda_{N-1}   &          & 2\lambda_N    & 1-\lambda_N  \\
        1-\lambda_N         & 1-\lambda_N       & \dots         & 1-\lambda_N   & 1 \\
    \end{pmatrix}
\]
for
\begin{align*}
    \lambda_i &= \frac{i}{2 N+1-i}, \qquad i=1,\dots,N.
\end{align*}

\paragraph{A recursive expression for the determinants}
We begin by deriving a recursion rule for the determinant of matrices of the following form:
\[
    M_k=
    \begin{pmatrix}
        d_0 &   a_1 &   a_2 &       \dots   &   a_{k-1} &   a_k \\
        a_1 &   d_1 &   a_2 &               &   a_{k-1} &   a_k \\
        a_2 &   a_2 &   d_2 &                   &   a_{k-1} &   a_k \\
        \vdots  &   &   &   \ddots  & & \vdots\\
        a_{k-1} &   a_{k-1} & a_{k-1}   &                   &   d_{k-1} &   a_k \\
        a_k &   a_k &   a_k &   \dots       &   a_k     &   d_k
    \end{pmatrix}.
\]

To find the determinant of $M_k$,
subtract the one before last row multiplied by $\frac{a_{k}}{a_{k-1}}$ from the last row: the last row becomes
\[
    (0,\dots,0,a_k-\frac{a_k}{a_{k-1}}d_{k-1},d_k-\frac{a_k}{a_{k-1}}a_{k}).
\]
Expanding the determinant along the last row we get
\[
    \det M_k = (d_k-\frac{a_k}{a_{k-1}}a_{k}) \det M_{k-1}-(a_k-\frac{a_k}{a_{k-1}}d_{k-1}) \det (M_k)_{k,k-1}
\]
where $(M_k)_{k,k-1}$ denotes the $k,k-1$ minor:
\[
    (M_k)_{k,k-1}=
    \begin{pmatrix}
        d_0 &   a_1 &   a_2 &       \dots   &   a_{k-2} &   a_k \\
        a_1 &   d_1 &   a_2 &               &   a_{k-2} &   a_k \\
        a_2 &   a_2 &   d_2 &                   &   a_{k-2} &   a_k \\
        \vdots  &   &   &   \ddots  \\
        a_{k-2} &   a_{k-2} & a_{k-2}   &                   &   d_{k-2} &   a_k \\
        a_{k-1} &   a_{k-1} & a_{k-1}   &                   &   a_{k-1} &   a_k
    \end{pmatrix}
    .
\]
If we multiply the last column of $(M_k)_{k,k-1}$ by $\frac{a_{k-1}}{a_k}$
we get a matrix that is different from $M_{k-1}$ by only the corner element.
Thus by basic determinant properties we get that
\[
    \frac{a_{k-1}}{a_k}\det (M_k)_{k,k-1} = \det M_{k-1}+(a_{k-1}-d_{k-1}) \det M_{k-2}.
\]
Combining these two results, we have found the following recursion rule for $\det M_k$, $k\geq 2$:
\begin{align*}
    &\det M_k = (d_k-\frac{a_k}{a_{k-1}}a_{k}) \det M_{k-1} \\&\qquad -(a_k-\frac{a_k}{a_{k-1}}d_{k-1})\left(\frac{a_k}{a_{k-1}} \det M_{k-1}+(a_{k}-\frac{a_k}{a_{k-1}}d_{k-1}) \det M_{k-2}\right) \\
            &= \left((d_k-\frac{a_k}{a_{k-1}}a_{k})-(a_k-\frac{a_k}{a_{k-1}}d_{k-1}) \frac{a_k}{a_{k-1}}\right) \det M_{k-1} -\left(a_k-\frac{a_k}{a_{k-1}}d_{k-1}\right)^2\det M_{k-2}
\end{align*}
or
\begin{equation}\label{E:mrecursion}
    \det M_k= \left(d_k-\frac{2 a_k^2}{a_{k-1}}+\frac{a_k^2 d_{k-1}}{a_{k-1}^2} \right) \det M_{k-1}-a_k^2\left(1-\frac{d_{k-1}}{a_{k-1}}\right)^2 \det M_{k-2} .
\end{equation}
Obviously, the recursion base cases are given by
\begin{align*}
    & \det M_0 = d_0, \\
    & \det M_1 = d_0 d_1-a_1^2.
\end{align*}

\paragraph{Closed form expressions for the determinants}
Going back to our matrix, $S_1$, by choosing
\begin{align*}
    & d_i = 2\frac{i+1}{2 N-i},\quad i=0,\dots,N-1 \\
    & d_N = 1 \\
    & a_i = \frac{i+1}{2 N-i}-\frac{i}{2 N+1-i},\quad i=1,\dots,N-1\\
    & a_N = 1-\frac{N}{N+1} = \frac{1}{N+1},
\end{align*}
we get that $M_k$ is the $k+1$'th leading principal minor of the matrix $S_1$.
The recursion rule~\eqref{E:mrecursion} can now be solved for this choice of $a_i$ and $d_i$. The solution is given by:
\begin{align}
    \det M_k =& \frac{(2N+1)^2}{(2N-k)^2} \left( 1 +\sum_{i=0}^k \frac{2N-2k-1}{2N+4 N i - 2 i^2+1}\right) \prod_{i=0}^{k} \frac{2N+4Ni-2i^2+1}{(2N+1-i)^2},  \label{E:mkdet}
\end{align}
for $k=0,\dots,N-1$, and
\begin{align}
    \det M_N = \det L_1
        =& \frac{(2N+1)^2}{(N+1)^2}\prod_{i=0}^{N-1} \frac{2N+4Ni-2i^2+1}{(2N+1-i)^2} \label{E:mndet}.
\end{align}

\paragraph{Verification}
We now proceed to verify the expressions \eqref{E:mkdet} and \eqref{E:mndet} given above. We will show that these expressions satisfy the recursion rule \eqref{E:mrecursion}
and the base cases of the problem. We begin by verifying the base cases:
\begin{align*}
    \det M_0 & = \frac{(2N+1)^2}{(2N)^2}\left( 1 +\frac{2N-1}{2N+1}\right) \frac{1}{2N+1}= \frac{1}{N} = d_0,
\end{align*}
\begin{align*}
    \det M_1 & = \frac{(2N+1)^2}{(2N-1)^2} \left( 1 +\frac{2N-3}{2N+1}+ \frac{2N-3}{6N -1}\right) \frac{1}{2N+1} \frac{6N-1}{(2N)^2}  \\
        &= \frac{28N^2-20N-1}{4N^2(2N-1)^2} = \frac{4}{N(2N-1)}-\left(\frac{2}{2N-1}-\frac{1}{2N}\right)^2= d_0 d_1-a_1^2.
\end{align*}
Now suppose $2 \leq k\leq N$. Denote
\begin{align*}
    \alpha_k
            &= d_k-\frac{2 a_k^2}{a_{k-1}}+\frac{a_k^2 d_{k-1}}{a_{k-1}^2} \\
            &=\begin{cases}
                4\frac{(2N+1)k-k^2-1}{(2N-k)^2},            &k<N\\
                3\frac{2N^2+2N-1}{(2N+1)^2},    &k=N
            \end{cases} \\
    \beta_k &= a_k^2\left(1-\frac{d_{k-1}}{a_{k-1}}\right)^2 \\
                &=\begin{cases}
                    \frac{(4 kN-2N-2k^2+4k-1)^2}{(2N-k)^2(2N-k+1)^2}, &k<N\\
                    \frac{(2N^2+2N-1)^2}{(N+1)^2(2N+1)^2}, &k=N
                \end{cases},
\end{align*}
then the recursion rule \eqref{E:mrecursion} can be written as
\begin{align*}
    & \det M_k= \alpha_k \det M_{k-1} - \beta_k \det M_{k-2}.
\end{align*}
Further denote
\begin{align*}
    f_i &= \frac{(2N+1)^2}{(2N-i)^2}, \quad i=0,\dots,N-1, \\
    g_i &= 2N-2i-1, \quad i=0,\dots,N-1,\\
    x_i &= \frac{1}{2N+4 N i - 2 i^2+1}, \quad i=0,\dots,N-1, \\
    y_i &= \frac{2N+4Ni-2i^2+1}{(2N+1-i)^2}, \quad i=0,\dots,N-1,
\end{align*}
then the solution \eqref{E:mkdet} becomes
\begin{align*}
    \det M_k =& f_k \left(1+g_k \sum_{i=0}^{k} x_i\right) \prod_{i=0}^{k} y_i,
\end{align*}
and \eqref{E:mndet} becomes
\begin{align*}
    \det M_N =& \frac{(2N+1)^2}{(N+1)^2} \prod_{i=0}^{N-1} y_i.
\end{align*}

Substituting \eqref{E:mkdet} in the RHS of \eqref{E:mrecursion} we get that for $k=2,\dots,N$
\begin{align*}
     & \alpha_k \det M_{k-1} - \beta_k \det M_{k-2}\\
        &= \alpha_k f_{k-1} \left(1+g_{k-1}\sum_{i=0}^{k-1} x_i\right) \prod_{i=0}^{k-1} y_i - \beta_k f_{k-2}\left(1+g_{k-2}\sum_{i=0}^{k-2} x_i \right) \prod_{i=0}^{k-2} y_i\\
        &= \left( \alpha_k f_{k-1} \left(1+g_{k-1} x_{k-1}+g_{k-1}\sum_{i=0}^{k-2} x_i\right)-\frac{\beta_k}{y_{k-1}} f_{k-2}-\frac{\beta_k}{y_{k-1}} f_{k-2} g_{k-2}\sum_{i=0}^{k-2} x_i \right) \prod_{i=0}^{k-1} y_i\\
        &= \left( \alpha_k f_{k-1}(1+ g_{k-1} x_{k-1}) -\frac{\beta_k}{y_{k-1}} f_{k-2}+\left( \alpha_k f_{k-1} g_{k-1}-\frac{\beta_k}{y_{k-1}} f_{k-2} g_{k-2}\right) \sum_{i=0}^{k-2} x_i\right) \prod_{i=0}^{k-1} y_i.
\end{align*}
It is straightforward (although somewhat involved) to verify that for $k<N$
\begin{align*}
    & \alpha_k f_{k-1} (1+ g_{k-1} x_{k-1}) -\frac{\beta_k}{y_{k-1}} f_{k-2} = f_k y_k (1+g_k x_{k-1}+g_k x_k),
\end{align*}
and
\begin{align*}
    & \alpha_k f_{k-1}g_{k-1}-\frac{\beta_k}{y_{k-1}} f_{k-2} g_{k-2} = f_k g_k y_k.
\end{align*}
We therefore get
\begin{align*}
    & \alpha_k \det M_{k-1} -\beta_k \det M_{k-2} \\
    & = \left( f_k y_k(1+g_k x_{k-1}+g_k x_k)+ f_k g_k y_k \sum_{i=0}^{k-2} x_i\right) \prod_{i=0}^{k-1} y_i\\
    & = f_k \left( 1+g_k \sum_{i=0}^{k} x_i\right) \prod_{i=0}^{k} y_i\\
    & = \det M_k,
\end{align*}
and thus \eqref{E:mkdet} satisfies \eqref{E:mrecursion}. It is also possible to show that
\begin{align*}
    & \alpha_N f_{N-1}(1+ g_{N-1} x_{N-1}) -\frac{\beta_N}{y_{N-1}} f_{N-2} = \frac{(2N+1)^2}{(N+1)^2}, \\
    & \alpha_N f_{N-1} g_{N-1}-\frac{\beta_N}{y_{N-1}} f_{N-2} g_{N-2} = 0,
\end{align*}
thus, for $k=N$
\begin{align*}
    & \alpha_N \det M_{N-1} -\beta_N \det M_{N-2} \\
    & = \frac{(2N+1)^2}{(N+1)^2} \prod_{i=0}^{N-1} y_i  \\
    & = \det M_N,
\end{align*}
and the expression \eqref{E:mndet} is also verified.

\comment{
Mathematica code to verify these statements:

\[Lambda][i_] := i/(2 n + 1 - i);
a[k_] := (\[Lambda][k + 1] - \[Lambda][k]);
d[k_] := 2 \[Lambda][k + 1];
d[n] = 1;
a[n] = 1 - \[Lambda][n];
\[Alpha][k_] := (d[k] - (2 a[k]^2)/a[k - 1] + (a[k]^2 d[k - 1])/a[k - 1]^2);
\[Beta][k_] := a[k]^2 (1 - d[k - 1]/a[k - 1])^2;
f[i_] := (2 n + 1)^2/(2 n - i)^2;
g[i_] := 2 n - 2 i - 1;
x[i_] := 1/(2 n + 4 n i - 2 i^2 + 1);
y[i_] := (2 n + 4 n i - 2 i^2 + 1)/(2 n + 1 - i)^2;

Simplify[\[Alpha][k] f[k - 1] (1 + g[k - 1] x[k - 1]) - \[Beta][k]/ y[k - 1] f[k - 2] - f[k] y[k] (1 + g[k] x[k - 1] + g[k] x[k])]
Simplify[\[Alpha][k] f[k - 1] g[k - 1] - \[Beta][k]/ y[k - 1] f[k - 2] g[k - 2] - f[k] g[k] y[k]]
Simplify[\[Alpha][n] f[n - 1] (1 + g[n - 1] x[n - 1]) - \[Beta][n]/ y[n - 1] f[n - 2] - (2 n + 1)^2/(n + 1)^2]
Simplify[\[Alpha][n] f[n - 1] g[n - 1] - \[Beta][n]/ y[n - 1] f[n - 2] g[n - 2]]
}

To complete the proof, note that the closed form expressions for $\det M_k$ consist of sums
and products of positive values, hence $\det M_k$ is positive, and thus follows from
Sylvester's criterion that $S_1$ is positive definite.

\bibliographystyle{plain}
\bibliography{bib}

\end{document}